\documentclass[11pt]{article}
\usepackage[dvips]{epsfig}
\usepackage{amsmath,amssymb,euscript,graphicx,amsfonts,verbatim}
\usepackage{enumerate,color}
\usepackage{graphicx}
\usepackage{enumitem}
\usepackage{enumerate,color}
\usepackage{graphicx}
\usepackage{soul,xcolor}

\usepackage[normalem]{ulem}
\newcommand\redout{\bgroup\markoverwith
{\textcolor{red}{\rule[.5ex]{2pt}{0.4pt}}}\ULon}
\usepackage{xcolor}

\usepackage[colorlinks=true,linkcolor=blue,urlcolor=blue,citecolor=blue]{hyperref}

\newtheorem{theorem}{Theorem}[section]
\newtheorem{proposition}[theorem]{Proposition}
\newtheorem{lemma}[theorem]{Lemma}
\newtheorem{corollary}[theorem]{Corollary}

\newtheorem{definition}[theorem]{Definition}
\newtheorem{problem}[theorem]{Problem}

\newtheorem{remark}[theorem]{Remark}
\newtheorem{notation}[theorem]{Notation}

\newenvironment{proof}{{\noindent \sc Proof.}}{\hfill $\Qed$\\}

\newcommand{\ths}{\theta^*}

\newcommand{\Qed}{\rule{2.5mm}{3mm}}

\newcommand{\Ga}{\Gamma}

\newcommand{\RR}{\mathbb{R}}

\newcommand{\G}{\Gamma}

\newcommand{\Es}{E^*}

\DeclareMathOperator{\mat}{Mat}

\newcommand{\MX}{\mat_X(\RR)}
\newcommand{\e}{E^{*}}

\newcounter{case}

\renewcommand{\thecase}{\arabic{case}}

\newcounter{subcase}

\numberwithin{subcase}{case}

\begin{document}


\begin{center}
{\bf\Large On the $Q$-polynomial property of  bipartite graphs admitting a uniform structure } \\ [+4ex]
Blas Fernández{\small$^{a,b,c}$},  
Roghayeh Maleki{\small$^{d}$},  
\v Stefko Miklavi\v c{\small$^{a, b,c}$},
\\
Giusy Monzillo{\small$^{a, b}$} 
\\ [+2ex]
{\it \small 
$^a$University of Primorska, UP IAM, Muzejski trg 2, 6000 Koper, Slovenia\\
$^b$University of Primorska, UP FAMNIT, Glagolja\v ska 8, 6000 Koper, Slovenia\\
$^c$IMFM, Jadranska 19, 1000 Ljubljana, Slovenia}\\
$^d$Department of Mathematics and Statistics, University of Regina, Regina, SK S4S 0A2, Canada
\end{center}


\begin{abstract}
	
	Let $\Ga$ denote a finite, connected graph with vertex-set $X$. Fix $x \in X$ and let $\varepsilon \ge 3$ denote the eccentricity of $x$.  For mutually distinct scalars $\{\ths_i\}_{i=0}^\varepsilon$ define a diagonal matrix $A^*=A^*(\ths_0, \ths_1, \ldots, \ths_{\varepsilon}) \in \MX$ as follows: for $y \in X$ we let $(A^*)_{yy} = \ths_{\partial(x,y)}$, where $\partial$ denotes the shortest path length distance function of $\Ga$. We say that $A^*$ is a {\em dual adjacency matrix candidate of} $\G$ {\em with respect to} $x$ if the adjacency matrix $A \in \MX$ of $\G$ and $A^*$ satisfy
	$$
	  A^3 A^* - A^* A^3+(\beta+1)( A A^* A^2 - A^2 A^* A)= \gamma(A^2A^*-A^*A^2)+\rho( A A^* - A^* A)  
	$$
	for some scalars $\beta, \gamma, \rho\in \RR$. 
	
	Assume now that $\G$ is uniform with respect to $x$ in the sense of Terwilliger [Coding theory and design theory, {P}art {I}, IMA Vol. Math. Appl., {\bf 20}, 193-212 (1990)]. In this paper, we give sufficient conditions on the uniform structure of $\Ga$, such that $\Ga$ admits a dual adjacency matrix candidate with respect to $x$. As an application of our results, we show that the full bipartite graphs of dual polar graphs are $Q$-polynomial.

\end{abstract}
\begin{quotation}

\end{quotation}

\begin{quotation}
\noindent 
{\em Math. Subj. Class.:}  05E99, 05C50.
\end{quotation}


\section{Introduction}
\label{sec:intro}

The $Q$-polynomial property of distance-regular graphs (and association schemes in general)  was introduced by Delsarte in his seminal work on coding theory \cite{Del}, and extensively investigated thereafter \cite{bannai2021algebraic, bannai-ito, BCN, DKT, Terpart1, TerpartII, Terpart3}. Let $\Ga$ be a $Q$-polynomial distance-regular graph. It turns out that for each vertex $x$ of $\Ga$ there exists a certain diagonal matrix $A^{*}= A^{*}(x)$, known as the {\em dual adjacency matrix of} $\Ga$ {\em with respect to} $x$. The eigenspaces of $A^{*}$ are the subconstituents of $\Ga$ with respect to $x$. The adjacency matrix $A$ of $\Ga$ and the dual adjacency matrix $A^{*}$ are related as follows: the matrix $A$ acts on the eigenspaces of  $ A^{*}$ in a (block) tridiagonal fashion, and $A^*$ acts on the eigenspaces of $A$ likewise \cite[Section 13]{ter2}. Only recently,  Terwilliger \cite{projective} used this property of $Q$-polynomial distance-regular graphs to extend the $Q$-polynomial property to graphs that are not necessarily distance-regular.  To do this, he dropped the assumption that every vertex of $\Ga$ has a dual adjacency matrix (instead he required that $\Ga$ has a dual adjacency matrix with respect to one distinguished vertex).

Let $\Ga=(X,{\mathcal E})$ denote a connected graph with vertex set $X$ and edge set ${\mathcal E}$. Let $\MX$ denote the matrix algebra over the real numbers $\RR$, consisting of all matrices with entries in $\RR$, whose rows and columns are indexed by $X$. Fix a vertex $x$ of $\Ga$ and let $\varepsilon \ge 3$ denote the eccentricity of $x$. Assume that there exists a dual adjacency matrix $A^* \in \MX$ of $\G$ with respect to $x$ (see Section \ref{sec:ter} for details). As  $A^*$ generates the dual adjacency algebra of $\G$ with respect to $x$, there exist mutually distinct scalars $\{\ths_i\}_{i=0}^\varepsilon$ such that the following holds: for every vertex $y$ of $\G$, the $(y,y)$-entry of $A^*$ is equal to $\ths_i$ if and only if $\partial(x,y)=i$, where $\partial$ denotes the shortest path length distance function of $\Ga$. Inspired by this property of a dual adjacency matrix, we introduce the following definition: for a set of mutually distinct scalars $\{\ths_i\}_{i=0}^\varepsilon$, we define $A^*=A^*(\ths_0, \ths_1, \ldots, \ths_{\varepsilon})$ to be a diagonal matrix in $\MX$, whose $(y,y)$-entry is equal to $\ths_{\partial(x,y)}$ for $y \in X$. If one wants to show that such a matrix indeed is a dual adjacency matrix of $\G$ with respect to $x$, it is often very convenient to first show that the adjacency matrix $A$ of $\G$ and $A^*$ satisfy the following equality for some scalars $\beta, \gamma, \rho$:
	\begin{equation}
		\label{eq:tridiag}
	A^3 A^* - A^* A^3+(\beta+1)( A A^* A^2 - A^2 A^* A)= \gamma(A^2A^*-A^*A^2)+\rho( A A^* - A^* A).
	\end{equation}
We will therefore say that $A^*$ is  a {\em dual adjacency matrix candidate} if there exist scalars $\beta, \gamma, \rho$, such that \eqref{eq:tridiag} holds.

Assume now that $\G$ is bipartite and uniform with respect to $x$ in the sense  of Terwilliger \cite{OldTer}. The main goal of this paper is to give sufficient conditions  on the uniform structure of $\Ga$ such that $\Ga$ admits a dual adjacency matrix candidate with respect to $x$. We provide such sufficient conditions in Theorems \ref{thm5} and \ref{thm6}. As an application of our results, we show that the full bipartite graphs of dual polar graphs are $Q$-polynomial. 
 
 Our paper is organized as follows. Sections~\ref{sec:ter} through \ref{sec:uniform} provide the necessary background information and highlight some interesting preliminary results. In Section~\ref{sec:tridiag}, we present our main results, including Theorems \ref{thm5} and \ref{thm6}. Finally, using the results from Section~\ref{sec:tridiag}, in Section~\ref{sec:app} we discuss an application related to the $Q$-polynomial property for the full bipartite graphs of dual polar graphs.


\section{Terwilliger algebras and $Q$-polynomial property}
\label{sec:ter}

Throughout this paper, all graphs will be finite, simple, connected, and undirected. Let $\Ga=(X,\mathcal{E})$ denote a graph with vertex set $X$ and edge set $\mathcal{E}$. In this section, we recall the $Q$-polynomial property of $\Ga$ and the definition of a Terwilliger algebra. Let $\partial$ denote the shortest path length distance function of $\Ga$. The diameter $D$ of $\Ga$ is defined as $D=\max\{\partial(x,y) \mid x,y \in X \}$. Let $\MX$ denote the matrix algebra over the real numbers $\RR$, consisting of all matrices with entries in $\RR$, whose rows and columns are indexed by $X$. Let $I \in \MX$ denote the identity matrix. Let $V$ denotes the vector space over $\RR$ consisting of all column vectors with entries in $\RR$, whose coordinates are indexed by $X$. We observe that $\MX$ acts on $V$ by left multiplication. We call $V$ the \emph{standard module}. For any $y \in X$, let $\widehat{y}$ denote the element of $V$ with $1$ in the ${y}$-coordinate and $0$ in all other coordinates. We observe that $\{\widehat{y}\;|\;y \in X\}$ is a basis for $V$. 

\begin{definition}
	\label{def:wam}
	By a weighted adjacency matrix of $\Ga$ we mean a matrix $A \in \MX$, which has $(z,y)$-entry 
\begin{eqnarray*}
	(A)_{z y} = \left\{ \begin{array}{lll}
		\ne 0 & \hbox{if } \; \partial(z,y)=1, \\
		0 & \hbox{if } \; \partial(z,y) \ne 1 \end{array} \right. 
	\qquad (y,z \in X).
\end{eqnarray*}
\end{definition}
For the rest of this section, we fix a weighted adjacency matrix $A$ of $\Ga$ that is diagonalizable over $\RR$. Let $M$ denote the subalgebra of $\MX$ generated by $A$. The algebra $M$ is called the \emph{adjacency algebra} of the graph $\Gamma$, generated by $A$. Observe that $M$ is commutative. Let ${\cal D}+1$ denote the dimension of the vector space $M$. Since $A$ is diagonalizable, the vector space $M$ has a
basis $\{E_i\}_{i=0}^{\cal D}$ such that $\sum_{i=0}^{\cal D} E_i=I$ and $E_i E_j = \delta_{i,j} E_i$ for $0 \le i, j \le {\cal D}$. We call $\{E_i\}_{i=0}^{\cal D}$ \emph{the primitive idempotents of} $A$. Since $A \in M$, there exist real numbers $\{\theta_i\}_{i=0}^{\cal D}$ such that $A = \sum_{i=0}^{\cal D} \theta_i E_i$. The scalars $\{\theta_i\}_{i=0}^{\cal D}$ are mutually distinct since $A$ generates $M$. We have
$A E_i = \theta_i E_i = E_i A$ for $0 \le i \le {\cal D}$. Note that
$$
  V = \sum_{i=0}^{\cal D} E_i V \qquad \qquad \text{(direct sum)}.
$$
For $0 \le  i \le  {\cal D}$  the subspace $E_i V$ is an eigenspace of $A$, and $\theta_i$ is the corresponding eigenvalue. For notational convenience, we assume $E_{i} = 0$ for $i < 0$ and $i > {\cal D}$.

Next we discuss the dual adjacency algebra of $\Ga$. To do that, we fix a vertex $x \in X$ for the rest of this section. Let $\varepsilon=\varepsilon(x)$ denote the eccentricity of $x$, that is, $\varepsilon = \max \{\partial(x,y) \mid y \in X\}$. For $ 0 \le i \le \varepsilon$, let $E_i^*=E_i^*(x)$ denote the diagonal matrix in $\MX$ with $(y,y)$-entry given by
\begin{eqnarray*}
	\label{den0}
	(\e_i)_{y y} = \left\{ \begin{array}{lll}
		1 & \hbox{if } \; \partial(x,y)=i, \\
		0 & \hbox{if } \; \partial(x,y) \ne i \end{array} \right. 
	\qquad (y \in X).
\end{eqnarray*}
We call $\e_i$ the \emph{$i$-th dual idempotent} of $\Gamma$ with respect to $x$ \cite[p.~378]{Terpart1}. We observe 
(ei)  $\sum_{i=0}^\varepsilon E_i^*=I$; 
(eii) $E_i^{*\top} = E_i^*$ $(0 \le i \le \varepsilon)$; 
(eiii) $E_i^*E_j^* = \delta_{ij}E_i^* $ $(0 \le i,j \le \varepsilon)$.
It follows that $\{E_i^*\}_{i=0}^\varepsilon$ is a basis for a commutative subalgebra $M^*=M^*(x)$ of $\MX$. The algebra $M^*$ is called the \emph{dual adjacency algebra} of $\Gamma$ with respect to $x$ \cite[p.~378]{Terpart1}. Note that for $0 \le i \le \varepsilon$ we have
$\e_i V = {\rm Span} \{ \widehat{y} \mid y \in X, \partial(x,y)=i\}$, 
and  
\begin{equation}
	\label{vsub}
	V = \sum_{i=0}^{\varepsilon} E_i^* V \qquad \qquad {\rm (direct\ sum}). \nonumber 
\end{equation}
The subspace $\e_i V$ is known as the \emph{$i$-th subconstituent of $\Gamma$ with respect to $x$}. 
For convenience we set $\e_{i}=0$ for $i<0$ and $i>\varepsilon$.
By the triangle inequality, for adjacent $y, z \in X$ the distances $\partial(x, y)$ and $\partial(x, z)$
differ by at most one. Consequently,
$$
  AE^*_iV \subseteq E^*_{i-1} V + E^*_i V + E^*_{i+1} V  \qquad (0 \le i \le \varepsilon). 
$$
Next we discuss the $Q$-polynomial property.

\begin{definition}
	\label{def:q1}
	{\rm (See \cite[Definition 20.6, Definition 20.7 and Definition 20.8]{ter2}.)} 
	A matrix $A^* \in \MX$ is called a dual adjacency matrix of $\Ga$ (with respect to $x$ and the ordering 
	$\{E_i\}_{i=0}^{\cal D}$ of the primitive idempotents) whenever $A^*$ generates $M^*$ and
	$$
	  A^* E_i V \subseteq E_{i-1} V + E_i V + E_{i+1} V \qquad (0 \le i \le {\cal D}). 
	$$
	We say that the ordering $\{E_i\}_{i=0}^{\cal D}$ is $Q$-polynomial with respect to $x$ whenever there exists a dual adjacency matrix of $\Ga$ with respect to $x$ and $\{E_i\}_{i=0}^{\cal D}$. We say that $A$ is $Q$-polynomial with respect to $x$ whenever there exists an ordering of the primitive idempotents of $A$ that is $Q$-polynomial with respect to $x$.
	\end{definition}

Assume that $\Ga$ has a dual adjacency matrix $A^*$ with respect to $x$ and $\{E_i\}_{i=0}^{\cal D}$. Since $A^* \in M^*$, there exist real numbers $\{\ths_i\}_{i=0}^\varepsilon$ such that $A^* = \sum_{i=0}^\varepsilon \ths_i E^*_i$. The scalars $\{\ths_i\}_{i=0}^\varepsilon$ are mutually distinct since $A^*$ generates $M^*$. We have $A^* E^*_i = \ths_i E^*_i = E^*_i A^*$ for $0 \le i \le \varepsilon$. 
We mentioned earlier that the sum $V = \sum_{i=0}^\varepsilon E^*_i V$ is direct. For $0 
\le i \le \varepsilon$ the subspace $E^*_i V$ is an eigenspace of $A^*$, and $\ths_i$ is the corresponding eigenvalue. As we investigate the $Q$-polynomial property, it is helpful to bring in the Terwilliger algebra
\cite{Terpart1,TerpartII,Terpart3}. The following definition is a variation on \cite[Definition 3.3]{Terpart1}.

The \emph{Terwilliger algebra of $\Gamma$ with respect to $x$ and $A$}, denoted by $T=T(x,A)$,  is the  subalgebra of $\MX$ generated by $M$ and $M^*$. If $\Ga$ has a dual adjacency matrix $A^*$ with respect to $x$ and $\{E_i\}_{i=0}^{\cal D}$, then $T$ is generated by $A$ and $A^*$, see \cite[Lemma 2.6]{projective}. The following results will also be useful.

\begin{lemma}
	\label{lem:products}
	{\rm (See \cite[Lemma 2.7]{projective}.)}
	We have $E^*_i  A E^*_j = 0$  if $|i - j| > 1 \; (0 \le i, j \le \varepsilon)$. 
	Assume that $\Ga$ has a dual adjacency matrix $A^*$ with respect to $x$ and $\{E_i\}_{i=0}^{\cal D}$. Then, $E_i A^* E_j = 0$ if $|i - j| > 1 \;
(0 \le i, j \le {\cal D})$.
\end{lemma}

By a \emph{$T$-module} we mean a subspace $W$ of $V$ such that $BW \subseteq W$ for every $B \in T$. Let $W$ denote a $T$-module. Then, $W$ is said to be {\em irreducible} whenever it is nonzero and contains no submodules other than $0$ and $W$.

Let $W$ denote an irreducible $T$-module. Observe that $W$ is a direct sum of the non-vanishing subspaces $E_i^*W$ for $0 \leq i \leq \varepsilon$. The \emph{endpoint} of $W$ is defined as $r:=r(W)=\min \{i \mid 0 \le i\le \varepsilon, \; \e_i W \ne 0 \}$, and the \emph{diameter} of $W$ as $d:=d(W)=\left|\{i \mid 0 \le i\le \varepsilon, \; \e_i W \ne 0 \} \right|-1 $. It turns out that $\e_iW \neq 0$ if and only if $r \leq i \leq r+d$ $(0 \leq i \leq \varepsilon)$, see \cite[Lemma 2.9]{projective}. The module $W$ is said to be \emph{thin} whenever $\dim(E^*_iW)\leq1$ for $0 \leq i \leq \varepsilon$. Two $T$-modules $W$ and $W^{\prime}$ are said to be  \emph{$T$-isomorphic} (or simply \emph{isomorphic}) whenever there exists a vector space isomorphism $\sigma: W \rightarrow W^{\prime}$ such that $\left( \sigma B - B\sigma \right) W=0$ for all $B \in T$. 


\section{A tridiagonal relation}

Recall our graph $\G=(X,\mathcal{E})$ from Section \ref{sec:ter}. To keep things simple we will assume for the rest of this paper that $A \in \MX$ is the usual adjacency matrix of $\G$:

\begin{eqnarray*}
	(A)_{z y} = \left\{ \begin{array}{lll}
		1 & \hbox{if } \; \partial(z,y)=1, \\
		0 & \hbox{if } \; \partial(z,y) \ne 1 \end{array} \right. 
	\qquad (y,z \in X).
\end{eqnarray*}

 Fix a vertex $x$ of $\Ga$ and let $\varepsilon \ge 3$ denote the eccentricity of $x$. Recall that if  there exists a dual adjacency matrix $A^* =A^*(x) \in \MX$ of $\G$ with respect to $x$, then there exist mutually distinct scalars $\{\ths_i\}_{i=0}^\varepsilon$ such that $A^* = \sum_{i=0}^\varepsilon \ths_i E^*_i$,  where $\Es_i \; (0 \le i \le \varepsilon)$ are the dual idempotents of $\G$ with respect to $x$. Motivated by this observation we introduce the following definition. 
  
  \begin{definition}
  	\label{def:diag}
  	For mutually distinct scalars $\{\ths_i\}_{i=0}^\varepsilon$ we let  $A^*=A^*(\ths_0, \ths_1, \ldots, \ths_{\varepsilon})$ to be a diagonal matrix in $\MX$ defined by
  	$$
  	  (A^*)_{yy} = \ths_i \; \hbox{ if and only if } \; \partial(x,y)=i \qquad (y \in X).
  	$$
  \end{definition}
Note that such a matrix by definition generates the dual adjacency algebra $M^*$ of $\Gamma$ with respect to $x$. Let $A^*$ be as in Definition \ref{def:diag}. It turns out that if one wants to show that $A^*$ is an actual dual adjacency matrix of $\Ga$, then it is usually helpful to show that $A, A^*$ satisfy 
\begin{equation}
	\label{3diag}
	A^3 A^* - A^* A^3+(\beta+1)( A A^* A^2 - A^2 A^* A)= \gamma(A^2A^*-A^*A^2)+\rho( A A^* - A^* A)
\end{equation}
for certain scalars $\beta, \gamma, \rho\in \RR$. Therefore, we will say that $A^*$ is a {\em dual adjacency matrix candidate with respect to} $x$, if there exist $\beta, \gamma, \rho\in \RR$, such that \eqref{3diag} is satisfied.  In this case, we refer to scalars  $\beta, \gamma, \rho$ as the {\em corresponding parameters} of the dual adjacency matrix candidate. We finish this section with some rather obvious results.

\begin{lemma}
	\label{lem:entries}
	Let $A^*$ be as in Definition \ref{def:diag}, and pick $y,z \in X$. Then, the following (i)--(iv) hold.
	\begin{itemize}
		\item[(i)] 
		$( A^3 A^* - A^* A^3)_{zy}=\gamma_3(y,z)(\theta^*_{\partial(x,y)}-\theta^*_{\partial(x,z)})$, where $\gamma_3(y,z)$ denotes the number of walks of length $3$ between $y$ and $z$.
		\item[(ii)] 
		$( A A^* A^2- A^2 A^* A)_{zy}=\displaystyle \sum_{[ y,v,w,z]} ( \theta^*_{\partial(x,w)}-\theta^*_{\partial(x,v)})$, where the sum is over all walks $[ y,v,w,z]$ between $y$ and $z$.
		\item[(iii)] 
		$( A^2 A^* - A^* A^2)_{zy}=\gamma_2(y,z)(\theta^*_{\partial(x,y)}-\theta^*_{\partial(x,z)})$, where $\gamma_2(y,z)$ denotes the number of walks of length $2$ between $y$ and $z$.
		\item[(iv)] 
		$( A A^* - A^* A)_{zy}$ is $ ( \theta^*_{\partial(x,y)}-\theta^*_{\partial(x,z)})$ if $\partial(z,y)=1$, and $0$ otherwise. 
	\end{itemize}
\end{lemma}

\begin{proof}
	Use elementary matrix multiplication, Definition \ref{def:diag}, and the fact that the $(v,w)$-entry of $A^i$ is equal to the number of walks of length $i$ between $v$ and $w$.
\end{proof}

\begin{theorem}
	\label{thm:gamma}
	Let $A^*$ be a dual adjacency matrix candidate of $\G$ with respect to $x$, and let $\beta, \gamma, \rho$ be the  corresponding parameters. If $\Ga$ is bipartite, then $\gamma=0$.
\end{theorem}
\begin{proof}
	Assume that $\Gamma$ is bipartite. Pick $y,z \in X$ such that $\partial(y,z)=2$ and $\partial(x,y)=\partial(x,z)+2$. Note that as $\varepsilon \ge 3$ such vertices $y,z$ exist. Observe that since $\G$ is bipartite, there are no walks of length $3$ between $y$ and $z$. Therefore, it follows from \eqref{3diag} and Lemma \ref{lem:entries} that 
	$$
	  \gamma \gamma_2(y,z) (\ths_{\partial(x,y)}-\ths_{\partial(x,z)})=0,
	$$
	where $\gamma_2(y,z)$ denotes the number of walks of length $2$ between $y$ and $z$. Note that $\gamma_2(y,z) \ne 0$ as $\partial(y,z)=2$. Note also that $\ths_{\partial(x,y)}-\ths_{\partial(x,z)} \ne 0$ as $\{\ths_i\}_{i=0}^{\varepsilon}$ are mutually distinct. It follows that $\gamma=0$.
\end{proof}


\section{The lowering, flat, and raising matrices}
\label{sec:lfr}

Recall our graph $\Ga=(X, {\cal E})$ from Section \ref{sec:ter}. Fix a vertex $x$ of $\Ga$,  and let $\varepsilon \ge 3$ denote the eccentricity of $x$. Let $\e_i \; (0 \le i \le \varepsilon)$ be the dual idempotents with respect to $x$, and let $T=T(x,A)$ denote the corresponding Terwilliger algebra. We now recall certain matrices in $T$.

\begin{definition} \label{def2} 
Considering the above mentioned notation, the matrices $L=L(x)$, $F=F(x)$, and $R=R(x)$ are defined as follows.
	\begin{eqnarray}\label{defLR}
		L=\sum_{i=1}^{\varepsilon}E^*_{i-1}AE^*_i, \hspace{1cm}
		F=\sum_{i=0}^{\varepsilon}E^*_{i}AE^*_i, \hspace{1cm}
		R=\sum_{i=0}^{\varepsilon-1}E^*_{i+1}AE^*_i. \nonumber 
	\end{eqnarray}
	We refer to $L$, $F$, and $R$ as the lowering, the flat, and the raising matrix with respect to $x$, respectively.
\end{definition}
Note that, by definition, $L, F, R \in T$, $F=F^{\top}$, $R=L^{\top}$, and $A=L+F+R$.
Observe that for $y,z \in X$ the $(z,y)$-entry of $L$ equals $1$ if $\partial(z,y)=1$ and $\partial(x,z)= \partial(x,y)-1$, and $0$ otherwise. The $(z,y)$-entry of  $F$ is equal to $1$ if $\partial(z,y)=1$ and $\partial(x,z)= \partial(x,y)$, and $0$ otherwise. Similarly, the $(z,y)$-entry of $R$ equals $1$ if $\partial(z,y)=1$ and $\partial(x,z)= \partial(x,y)+1$, and $0$ otherwise. Consequently, we have
\begin{equation}
	\label{eq:LRaction}
	L \e_i V \subseteq \e_{i-1} V, \qquad  F \e_i V \subseteq \e_{i} V, \qquad R \e_i V \subseteq \e_{i+1} V.
\end{equation}
Observe also that $\Gamma$ is bipartite if and only if $F=0$. We now recall a connection between matrices $R,F,L$ and certain walks in $\Ga$. The concept of the shape of a walk with respect to $x$ was introduced in~\cite{FerMik}. Now, we recall the definition of this concept together with some immediate consequences.
\begin{definition}\label{def3}
	Let $\Gamma = (X , \mathcal{E})$ be a graph and fix a vertex $x$ of $\Ga$. Pick $y, z \in  X$, and let
	$P = [y = x_0, x_1, . . . , x_j = z]$ denote a $yz$-walk.  The shape of $P$ with respect to $x$ is a sequence of symbols $t_1 t_2  \ldots t_j$, where for $1 \leq i \leq j$ we have $t_i \in \{\ell, f, r\}$, and such that $t_i = r$ if $\partial(x,x_i) = \partial(x,x_{i-1}) + 1$,  $t_i = f$ if $\partial(x,x_i) = \partial(x,x_{i-1}) $, and $t_i =  \ell$ if $\partial(x,x_i)  = \partial(x,x_{i-1})-1$.
	The number of $yz$-walks of the shape $t_1t_2... t_j $ with respect to $x$ will be denoted as $t_1t_2... t_j(y,z)$. We will use exponential notation for the shapes containing several consecutive identical symbols.
\end{definition}

The following Lemma is a consequence of elementary matrix multiplication and comments below the Definition \ref{def2} (see also \cite[Lemma 4.2]{FerMik}).

\begin{lemma}
	\label{lem:walks}
	With reference to the notation above, pick $y,z\in X$. Then, the following hold for a positive integer $m$.
	\begin{enumerate}[label=(\roman*),]
		\item The $(z,y)$-entry of $L^m$ is equal to the number $\ell^m(y,z)$ with respect to $x$.
		\item The $(z,y)$-entry of $L^mR$ is equal to the number $r\ell^m(y,z)$ with respect to $x$.
		\item The $(z,y)$-entry of $RL^m$ is equal to the number $\ell^m r(y,z)$ with respect to $x$.
		\item The $(z,y)$-entry of $LRL$ is equal to the number $\ell r\ell(y,z)$ with respect to $x$.
	\end{enumerate}
\end{lemma}

\section{The uniform structure of a bipartite graph}
\label{sec:uniform}

For the rest of this paper assume that our graph $\Ga$ is bipartite. In this section, we discuss the uniform property of bipartite graphs. The uniform property was first defined for graded partially ordered sets \cite{OldTer}. The definition was later extended to bipartite distance-regular graphs in \cite{MikTer} and then to an arbitrary bipartite graph in \cite{FMMM}. For the ease of reference we recall the definition. Let $\Gamma=(X, \mathcal{R})$ denote a bipartite graph and let $V$ denote the standard module of $\Ga$. Fix $x\in X$, and let $\varepsilon \ge 3$ denote the eccentricity of $x$. Let $T$, $L$, and $R$ denote the corresponding Terwilliger algebra, lowering, and raising matrix, respectively.
\begin{definition}\label{3Diag}
	A \emph{parameter matrix} $U=(e_{ij})_{1\leq i,j\leq \varepsilon}$ is defined to be a tridiagonal matrix with entries in $\RR$, satisfying the following properties:
	\begin{enumerate}[label=(\roman*)]
		\item $e_{ii}=1$ $(1\leq i\leq \varepsilon)$,
		\item $e_{i, i-1}\neq0$ for $2 \le i \le \varepsilon$ or $e_{i-1,  i}\neq0$ for $2\leq i\leq \varepsilon$, and
		\item  the principal submatrix $(e_{ij})_{s\leq i, \, j\leq t}$ is nonsingular for $1\leq s\leq t\leq \varepsilon$.
	\end{enumerate}
	For convenience we write $e^{-}_i:=e_{i, i-1}$ for $2\leq i\leq \varepsilon$ and $e^{+}_i:=e_{i, i+1}$ for $1\leq i\leq \varepsilon-1$. We also define $e^{-}_1:=0$ and $e^{+}_\varepsilon:=0$. 
\end{definition}

Let $U=(e_{ij})_{1\leq i,j\leq \varepsilon}$ be as in Definition \ref{3Diag}. Observe that since the principal submatrix $(e_{ij})_{s\leq i, \, j\leq s+1}$ is nonsingular for $1\leq s\leq \varepsilon-1$, we have that $e_{i+1}^- e_i^+ \ne 1$ for $1 \le i \le \varepsilon-1$. A \emph{uniform structure} of $\Gamma$ with respect to $x$ is a  pair $(U,f)$ where $f=\{f_i\}_{i=1}^\varepsilon$ is a vector in $\RR^\varepsilon$, such that
\begin{align}\label{uniformeq}
	e^{-}_iRL^2+LRL+ e^+_i L^2R=f_iL
\end{align}
is satisfied on $E^*_iV$ for $1\leq i\leq \varepsilon$, where $\e_i\in T$ are the dual idempotents of $\Gamma$ with respect to $x$. If the vertex $x$ is clear from the context, we will simply use \emph{uniform structure} of $\Gamma$ instead of  \emph{uniform structure of $\Gamma$ with respect to $x$}. 
The following result by Terwilliger plays an important role in the rest of this paper. 

\begin{theorem}[{\cite[Theorem 2.5]{OldTer}}]\label{oldpaper}
	Let $\Gamma=(X, \mathcal{R})$ denote a bipartite graph and fix $x\in X$. Let $T=T(x)$ denote the corresponding Terwilliger algebra. Assume that $\Gamma$ admits a uniform structure $(U,f)$ with respect to $x$. Then, the following (i)--(iii) hold:
	\begin{enumerate}[label=(\roman*),]
		\item Every irreducible $T$-module is thin. 
		\item Let $W$ denote an irreducible $T$-module with endpoint $r$ and diameter $d$. Then, the isomorphism class of $W$ is determined by $r$ and $d$. 
		\item[(iii)]  $W$ has a basis $\{w_i\}_{i=r}^{r+d}$ such that the following hold:
		\begin{itemize}
			\item[a)] $w_i \in \e_i W \; (r \le i \le r+d)$; 
			\item[b)] $Lw_r=0$ and $ L w_{r+i} = w_{r+i-1} \; (1\le i \le d)$;
			\item[c)] There exist scalars $x_{r+i} \; (1 \le i \le d)$ such that $Rw_{r+i-1} = x_{r+i} w_{r+i}$, and $R w_{r+d}=0$. Moreover, $\{x_{r+i}\}_{i=1}^{d}$  is the solution to the linear system 
			\begin{equation}
				\label{eq:linsys}
				U(r,d) 
				\begin{pmatrix}
					x_{r+1} \\
					 x_{r+2}\\
					 \vdots \\
					x_{r+d}
				\end{pmatrix}
				=
				\begin{pmatrix}
					f_{r+1} \\
					f_{r+2}\\
					\vdots \\
					f_{r+d}
				\end{pmatrix},
			\end{equation}
			where $U(r,d) = (e_{ij})_{r+1\leq i, \, j\leq d}$.
		\end{itemize}
	\end{enumerate}
\end{theorem}


\section{Uniform graphs and a tridiagonal relation}
\label{sec:tridiag}

Recall our uniform graph $\G$ from Section \ref{sec:uniform}. In this section we display sufficient conditions on the pair $(U,f)$ of the corresponding uniform structure for $\G$ to have a dual adjacency matrix candidate with respect to $x$. For the rest of the paper we adopt the following notation.

\begin{notation}
	\label{not}
	Let $\Ga=(X, {\cal E})$ denote a finite, simple, and connected bipartite graph with vertex set $X$ and edge set ${\cal E}$.  Let $A \in \MX$ be the adjacency matrix of $\Ga$.  Pick $x\in X$ and let $\varepsilon=\varepsilon(x) \ge 3$ denote the eccentricity of $x$. Let $\Es_i \; (0 \le i \le \varepsilon)$ be the dual idempotents of $\G$ with respect to $x$. Let $L=L(x)$ and $R=R(x)$ be the lowering and the raising matrix of $\Ga$ with respect to $x$. Assume that $\Ga$ admits a uniform structure with respect to $x$, and let $e_i^- \; (2 \le i \le \varepsilon)$, $e_i^+ \; (1 \le i \le \varepsilon-1)$, $f_i \; (1\le i \le \varepsilon)$ be as defined in Section \ref{sec:uniform}. For $y,z \in X$ and a positive integer $i$ let $\gamma_i(y,z)$ denote the number of walks of length $i$ between $y$ and $z$. In addition, for mutually distinct scalars $\{\ths_i\}_{i=0}^{\varepsilon}$ let $A^*=A^*(\ths_0, \ths_1, \ldots, \ths_{\varepsilon})$ denote a diagonal matrix in $\MX$ defined by 
	$$
	(A^*)_{yy}=\theta^*_{\partial(x,y)} \qquad (y \in X).
	$$
\end{notation}

\begin{theorem}\label{thm5}
	With reference to Notation~\ref{not}, pick mutually distinct scalars $\{\ths_i\}_{i=0}^{\varepsilon}$ and let $A^*=A^*(\ths_0, \ths_1, \ldots, \ths_{\varepsilon}) $.
	Assume that there exist scalars $\beta, \rho \in \RR$ such that the following equalities are satisfied:
	\begin{align}
	\beta+1&=\frac{\ths_{i-2}-\ths_{i+1}}{\ths_{i-1}-\ths_{i}}  \qquad &(2\le i \le \varepsilon-1),  \label{meq1} \\[.1in]
		e_{i}^-(\beta+2)&= 1+(\beta+1)\frac{\ths_{i-1}-\ths_{i-2}}{\ths_{i-1}-\ths_{i}} \qquad &(2\le i \le {\varepsilon}),\label{meq2} \\[.1in]
		e_i^+(\beta+2)&=1+(\beta+1)\frac{\ths_{i}-\ths_{i+1}}{\ths_i-\ths_{i-1}} \qquad &(1\le i \le {\varepsilon-1}), \label{meq4} \\[.1in]
		f_i(\beta+2)&=\rho \qquad &(1\le i \le {\varepsilon}).\label{meq6}.
	\end{align}
Then $A^*$ is a dual adjacency matrix candidate with respect to $x$ and with corresponding parameters $\beta, \rho$.
\end{theorem}

\begin{proof}
	By Theorem \ref{thm:gamma}, it suffices to show that the equation 
	\begin{equation}\label{3diagg}
	A^3 A^* - A^* A^3+(\beta+1)( A A^* A^2 - A^2 A^* A)= \rho( A A^* - A^* A)
\end{equation}
holds. We will do that by comparing the $(z,y)$-entry of the left and the right side of \eqref{3diagg} for $y,z \in X$. It is clear from Lemma \ref{lem:entries} that, if $\partial(z,y) \ge 4$, then the $(z,y)$-entries of both sides of \eqref{3diagg} are equal to $0$. Next, if $z=y$ or $\partial(z,y)=2$, then there are no walks of length $3$ between $y$ and $z$ (recall that $\Ga$ is bipartite), and so again, by Lemma \ref{lem:entries}, the $(z,y)$-entries of both sides of  \eqref{3diagg} are equal to $0$.  In what follows, we are left with the case $\partial(y,z) \in \{1,3\}$. Note that as $\Ga$ is bipartite we have that $\partial(x,y) \ne \partial(x,z)$. Therefore, we will assume that $\partial(x,z) < \partial(x,y)$; if $\partial(x,z) > \partial(x,y)$, then the proof is analogous. We split our study into two cases. 

\vspace{.1in}
\noindent
{\bf Case 1:} $\partial(z,y)=3$ and $\partial(x,z)=i-2, \ \partial(x,y)=i+1$ for some $2\leq i \leq \varepsilon-1$.

\vspace{.1in}

\noindent 
Note that if $[ y,v,w,z]$ is a walk of length 3 between $y$ and $z$, then we have that $\partial(x,v)=i$ and $\partial(x,w)=i-1$. Therefore, using Lemma \ref{lem:entries}, we get that 
	\begin{align*}
(A^3 A^* - A^* A^3)_{zy} &= \gamma_3(y,z)(\theta^*_{i+1}-\theta^*_{i-2}),\\
(A A^* A^2 - A^2 A^* A)_{zy} &= \gamma_3(y,z)(\theta^*_{i-1}-\theta^*_{i}),\\
(AA^*-A^*A)_{zy}&=0.
	\end{align*}
	
Using \eqref{meq1} we find that $(z,y)$-entries of both sides of \eqref{3diagg} agree.

\vspace{.1in}
\noindent
{\bf Case 2:} $\partial(x,z)=i-1, \ \partial(x,y)=i$ for some $1 \leq i \leq \varepsilon$.

\vspace{.1in}

\noindent 
Note that in this case we have exactly three possible shapes of walks of length $3$ between $y$ and $z$ in $\Ga$, namely walks of shapes $\ell^2 r$, $\ell r \ell$, and $r \ell^2$. Let us abbreviate $a=\ell^2r(y,z)$, $b=\ell r \ell(y,z)$, $c=r \ell^2(y,z)$, and note that $a=0$ ($c=0$, respectively) if $i=1$ ($i=\varepsilon$, respectively). Note that by Lemma \ref{lem:walks} we have $a = (RL^2)_{zy}$, $b=(L R L)_{zy}$, and $c=(L^2 R)_{zy}$. Let $[ y,v,w,z]$ be a walk of length 3 between $y$ and $z$. Observe that if $[ y,v,w,z]$  is of shape $\ell^2 r$, then $\partial(x,v)=i-1$ and $\partial(x,w)=i-2$.  If $[ y,v,w,z]$  is of shape $\ell r \ell$, then $\partial(x,v)=i-1$ and $\partial(x,w)=i$. Similarly,  if $[ y,v,w,z]$  is of shape $r \ell^2$, then $\partial(x,v)=i+1$ and $\partial(x,w)=i$. Therefore, using Lemma \ref{lem:entries}, we get that
\begin{align*}
	(A^3 A^* - A^* A^3)_{zy} &= (a+b+c)(\ths_{i}-\ths_{i-1}),\\
	(A A^* A^2 - A^2 A^* A)_{zy} &= a(\ths_{i-2}-\ths_{i-1}) + b(\ths_i-\ths_{i-1}) + c(\ths_i-\ths_{i+1}),\\
	(AA^*-A^*A)_{zy}&=\left\{ \begin{array}{lll}
		\ths_i-\ths_{i-1} & \hbox{if } \; \partial(y,z)=1, \\
		0 & \hbox{if } \; \partial(y,z) =3. \end{array} \right. 
\end{align*} 

Recall also that since $\G$ admits a uniform structure, the following equation is satisfied on $E^*_iV$ for $1\leq i\leq \varepsilon$:
\begin{equation}\label{unieq}
		e_i^-RL^2+LRL+ e_i^+ L^2R=f_iL.
\end{equation}
Computing the $(z,y)$-entry of the above equation using Lemma \ref{lem:walks} yields 
	\begin{eqnarray}\label{entrywise}
		e_i^-a+b+e_i^+c= \left\{ \begin{array}{lll}
			f_i & \hbox{ if } \; \partial(z,y)=1, \\
			0  &  \hbox{ otherwise. } \end{array} \right. 
	\end{eqnarray}

It is now clear that the $(z,y)$-entries of \eqref{3diagg} agree if  \eqref{meq2}, \eqref{meq4} nad \eqref{meq6} hold.
\end{proof}

With reference to Notation~\ref{not} and the observation after Definition \ref{3Diag}, it follows that $e_{i+1}^- e_i^+ \ne 1$ for $1 \le i \le \varepsilon-1$. In the next result we derive some further restrictions on parameters $\beta$, $e_i^-$ and $e_i^+$, that follow only from equalities \eqref{meq2} and \eqref{meq4} given in Theorem \ref{thm5}.

\begin{theorem}
	\label{thm6}
	With reference to Notation~\ref{not}, pick a scalar $\beta \in \RR$ and mutually distinct scalars $\{\ths_i\}_{i=0}^{\varepsilon}$. Then, the following (i), (ii) are equivalent:
	\begin{itemize}
		\item[(i)] 
		the equalities \eqref{meq2} and \eqref{meq4} from Theorem \ref{thm5} hold;
		\item[(ii)]
		we have $\beta \in \RR \setminus \{-2,-1\}$, $e_i^- \ne 1$ for $2 \le i \le \varepsilon$, $e_i^+ \ne 1$ for $1 \le i \le \varepsilon-1$, and the following (iia), (iib) hold for $1 \le i \le \varepsilon-1$:
		\begin{itemize}
			\item[(iia)]
			\begin{equation}
				\label{eq6a}
				\beta+1=\frac{(e_{i+1}^--1)(e_i^+-1)}{1-e_i^+ e_{i+1}^-};
			\end{equation}
		\item[(iib)]
		\begin{equation}
			\label{eq6b}
			\ths_{i+1} = \ths_1+(\ths_1-\ths_0) \sum_{j=1}^{i} (-1)^j \prod_{k=1}^j F(k),
		\end{equation}
	where $F(i)=-(e_i^+-1)/(e_{i+1}^- - 1)$ for $1 \le i \le \varepsilon-1$. 
		\end{itemize}
	\end{itemize}
\end{theorem}

\begin{proof}
  Assume first that (i) above holds. Suppose that $\beta=-1$. Then \eqref{meq2} and \eqref{meq4} yield that $e_1^+=e_2^-=1$, contradicting $e_1^+ e_2^- \ne 1$.  Suppose that $\beta=-2$. Then, \eqref{meq4} yields $\ths_0 = \ths_2$, a contradiction. Suppose that $e_i^- = 1$ for some $2 \le i \le \varepsilon$. Then, employing \eqref{meq2} we get that
  $$
  \beta+1 = (\beta+1) \frac{\ths_{i-1} - \ths_{i-2}}{\ths_{i-1} - \ths_i}.
  $$
  As $\beta \ne -1$, this yields $\ths_i = \ths_{i-2}$, a contradiction. This shows that $e_i^- \ne 1$.
  Similarly as above we show that $e_i^+ \ne 1$ for $1 \le i \le \varepsilon-1$ .

  Pick an integer $i \; (1 \le i \le \varepsilon-1)$. Observe that $e_i^+(\beta+2)-1 \ne 0$, otherwise we would have $\ths_i = \ths_{i+1}$, a contradiction. Note that it follows from \eqref{meq2} and \eqref{meq4} that
  $$
  \frac{e_{i+1}^-(\beta+2)-1}{\beta+1} = \frac{\beta+1}{e_i^+(\beta+2)-1}.
  $$
  Solving this equation for $\beta+1$ (and keeping in mind that $\beta\neq -2$), we obtain \eqref{eq6a}.  Using  \eqref{eq6a} we now easily find that 
  $$
    F(i)=\frac{e_i^+(\beta+2)-1}{\beta+1},
  $$
  and so using \eqref{meq4} we have 
  $$
    F(i) = \frac{\ths_i-\ths_{i+1}}{\ths_i-\ths_{i-1}} \qquad (1 \le i \le \varepsilon-1).
  $$
  It follows
  $$
F(i) F(i-1) \cdots F(1)   = \frac{\ths_i-\ths_{i+1}}{\ths_i-\ths_{i-1}} \frac{\ths_{i-1}-\ths_i}{\ths_{i-1}-\ths_{i-2}} \cdots \frac{\ths_1-\ths_2}{\ths_1-\ths_0}   = 
 (-1)^{i-1} \frac{\ths_i-\ths_{i+1}}{\ths_1-\ths_0},
  $$
  and so 
  \begin{equation}
  	\label{eq6c}
  	\ths_{i+1} - \ths_i = (-1)^{i} (\ths_1-\ths_0) F(i) F(i-1) \cdots F(1).
  \end{equation}
Adding equations \eqref{eq6c} we finally get \eqref{eq6b}.

Assume now that (ii) above holds.  Note that it follows from \eqref{eq6a} that 
$$
\frac{e_{i+1}^-(\beta+2)-1}{\beta+1} = \frac{\beta+1}{e_i^+(\beta+2)-1},
$$
and so \eqref{meq2} holds if and only if  \eqref{meq4} holds. Let us therefore prove that  \eqref{meq4} holds. We have 
$$
\frac{\ths_{i}-\ths_{i+1}}{\ths_i-\ths_{i-1}} = \frac{-(-1)^i \prod_{k=1}^i F(k)}{(-1)^{i-1} \prod_{k=1}^{i-1} F(k)} =F(i) = -(e_i^+-1)/(e_{i+1}^- - 1) = \frac{e_i^+(\beta+2)-1}{\beta+1}.
$$
This finishes the proof.
\end{proof}

\begin{theorem}
	\label{thm7}
	With reference to Notation~\ref{not}, pick a scalar $\beta \in \RR$ and mutually distinct scalars $\{\ths_i\}_{i=0}^{\varepsilon}$. Assume that equivalent conditions (i), (ii) of Theorem \ref{thm6} hold. Then, the following (i), (ii) are equivalent:
	\begin{itemize}
		\item[(i)] 
		the equality \eqref{meq1} from Theorem \ref{thm5} hold  for $2 \le i \le \varepsilon-1$;
		\item[(ii)]
		\begin{equation}
			\label{eq7a}
			1+F(i-1) F(i)=-\beta F(i-1)
		\end{equation}
		hold for $2 \le i \le \varepsilon-1$. 
	\end{itemize}
\end{theorem}

\begin{proof}
	Using \eqref{eq6b} we easily find that \eqref{meq1} and \eqref{eq7a} are equivalent for $2 \le i \le \varepsilon-1$.
\end{proof}

\begin{remark}\label{check}
	With reference to Notation~\ref{not}, observe that Theorems \ref{thm5}, \ref{thm6} and \ref{thm7} gives us the following algorithm for finding a dual adjacency matrix candidate with respect to $x$.
	\begin{enumerate}
		\item
		Check if there exists a scalar $\beta \in \RR \setminus \{-2,-1\}$ such that \eqref{eq6a} hold for $1 \le i \le \varepsilon-1$.
		\item 
		If the above check is positive, then pick scalars $\ths_0, \ths_1 \; (\ths_0 \ne \ths_1)$, and set 
		$$
			\ths_{i+1} = \ths_1+(\ths_1-\ths_0) \sum_{j=1}^{i} (-1)^j \prod_{k=1}^j F(k)
	     $$
	     for $1 \le i \le \varepsilon-1$, where $F(i)=-(e_i^+-1)/(e_{i+1}^- - 1)$ for $1 \le i \le \varepsilon-1$. 
	     \item
	     Check if the scalars $\{\ths_i\}_{i=0}^{\varepsilon}$ are mutually distinct.
	     \item
	     Check if \eqref{eq7a} holds for $2 \le i \le \varepsilon-1$.
	     \item
	     Check if the scalars $ \{f_i\}_{i=0}^{\varepsilon}$ are constant and set $\rho=f_i(\beta+2)$.
	     \item
	     If all of the above checks are positive, then it follows from Theorems \ref{thm5}, \ref{thm6} and \ref{thm7} that $A^*=A^*(\ths_0, \ths_1, \ldots, \ths_{\varepsilon}) $ is a dual adjacency matrix candidate with respect to $x$ with corresponding parameters $\beta, \rho$.
	\end{enumerate}
\end{remark}

\begin{remark}
With reference to Theorems \ref{thm6} and \ref{thm7}, assume $\ths_0~=~-1$, $\ths_1~=~0$, and set 
$$
P(i)=(-1)^iF(i)F(i-1)\cdots F(1) \qquad (1\le i \le \varepsilon-1)
$$ 
with $P(0)=1$. Then \eqref{eq6b} yields
\begin{equation}
			\label{eq6bb}
			\ths_{i} = \sum_{j=1}^{i-1} (-1)^j \prod_{k=1}^j F(k)=\sum_{j=1}^{i-1}P(j) \qquad (2\le i \le \varepsilon).
\end{equation}
On the other hand, multiplying both sides of \eqref{eq7a} by $P(i-2)$, we obtain a recurrence formula for $P(i)$:
\begin{equation}\label{eq7b}
P(i)=\beta P(i-1)-P(i-2) \qquad (2 \le i \le \varepsilon-1),
\end{equation}
with $P(0)=1$ and $P(1)=-F(1)$. Thus $P(i)$ is a constant-recursive sequence of degree $2$. Here we need to distinguish two cases: $(i) \ \beta=2$ or $(ii) \ \beta\neq 2$.

$(i) \ \beta=2$. The characteristic roots of the recurrence \eqref{eq7b} are identical and equal to $1$. Then, using the initial values of the sequence, we obtain
$$
P(i)=1-(1+F(1))i \qquad (0 \le i \le \varepsilon-1),
$$
which is an arithmetic progression with first term $1$ and common difference $-(1+F(1))$. Thus \eqref{eq6bb} gives
$$
\ths_i=\frac{1}{2}(i-1)\left(2 - (1 + F(1))i\right) \qquad (0\le i \le \varepsilon).
$$

$(ii) \ \beta\neq 2$. Since the characteristic roots of the recurrence \eqref{eq7b} are $\frac{\beta \pm\sqrt{\beta^2-4}}{2}$, using the initial values of the sequence, it follows that
$$
P(i)=\left(\frac{1}{2}-\frac{\beta+2F(1)}{2\sqrt{\beta^2-4}}\right) \left(\frac{\beta +\sqrt{\beta^2-4}}{2}\right)^i+\left(\frac{1}{2}+\frac{\beta+2F(1)}{2\sqrt{\beta^2-4}}\right)\left(\frac{\beta -\sqrt{\beta^2-4}}{2}\right)^i \ (0 \le i \le \varepsilon-1),
$$
which is the sum of two distinct geometric progressions with first term $\frac{1}{2}\mp\frac{\beta+2F(1)}{2\sqrt{\beta^2-4}}$ and common ratio $r_{\pm}=\frac{\beta \pm\sqrt{\beta^2-4}}{2}$, respectively. Thus, \eqref{eq6bb} becomes
\begin{equation*}
\ths_{i} = \frac{1+r_+ F(1)}{\left(\beta-2\right) \left(1+r_+ \right)} (1-r_+^{i-1}) +
\frac{1+r_{-}F(1)}{\left(\beta-2\right) \left(1+r_{-}\right)} \left(1-r_{-}^{i-1}\right) \qquad (0\le i \le \varepsilon).
\end{equation*}

\end{remark}


\section{An application: The full bipartite graphs of  dual polar graphs}\label{sec:app}

In this section, we use the results from Section \ref{sec:tridiag} to show that the full bipartite graphs of dual polar graphs are $Q$-polynomial. Let us first recall the definition of a full bipartite graph. 

\begin{definition}\label{def3.1}
	Let $\Delta=(Y,\mathcal{F})$ denote a graph with vertex set $Y$ and edge set $\mathcal{F}$. Fix $x \in Y$ and define $\mathcal{F}_f = \mathcal{F} \setminus \{yz \mid \partial(x,y) = \partial(x,z)\}$. Observe that the graph $\Delta_f=(Y,\mathcal{F}_f)$ is bipartite. The graph $\Delta_f $ is called the full bipartite graph of $\Delta$ with respect to $x$. 
\end{definition}
With reference to Definition \ref{def3.1}, let $\varepsilon=\varepsilon(x)$ denote the eccentricity of $x$ and let $V$ denote the standard module for $\Delta$. Since the vertex set of $\Delta$ is equal to the vertex set of $\Delta_f$, observe that $V$ is also the standard module for $\Delta_f$. Recall that the Terwilliger algebra $T=T(x)$ of $\Delta$ is generated by the adjacency matrix $A$ and the dual idempotents $\e_i$ ($0\leq i \leq \varepsilon$). Furthermore, we have $A=L+F+R$ where $L,F$, and $R$ are the corresponding lowering, flat, and raising matrices, respectively. Let $A_f$ denote the adjacency matrix of $\Delta_f$ and let $T_f=T_f(x,A_f)$ be the Terwilliger algebra of $\Delta_f$ with respect to $x$. As $\Delta_f$ if bipartite, the flat matrix of $\Delta_f$ with respect to $x$ is equal to the zero matrix. Moreover, the lowering and the raising matrices of $\Delta_f$ with respect to $x$ are equal to $L$ and $R$, respectively. It follows that $A_f=L+R$.  For $0 \le i \le \varepsilon$, note also that the $i$-th dual idempotent of $\Delta_f$ with respect to $x$ is equal to $\e_i$. Consequently, the algebra $T_f$ is generated by the matrices $L, R$, and $\e_i$ ($0\leq i \leq \varepsilon$). Furthermore, the dual adjacency algebra of $\Delta$ with respect to $x$ coincides with the dual adjacency algebra of $\Delta_f$ with respect to $x$.

\begin{lemma} 
	(\cite[Lemma 3.4]{FMMM})
	\label{lem:modules2}
	With the above notation, let $W$ denote a $T$-module. Then,  the following (i), (ii) hold.

	\begin{enumerate}[label=(\roman*)]
		\item $W$ is a $T_f$-module.
		\item  If $W$ is a thin irreducible $T$-module, then $W$ is a thin irreducible $T_f$-module.
	\end{enumerate}
\end{lemma}

Let $b$ denote a prime power and let $\mathbb{F}_b$ denote the finite field of order $b$. Let $U$ denote a finite-dimensional vector space over $\mathbb{F}_b$ endowed with one of the following nondegenerate forms.
$$
\begin{array}{l|l|l|l}
\text { name } & \operatorname{dim} U & \text { form } & {e} \\
\hline C_D(b) & 2 D & \text { symplectic } & 1 \\[.05in]
B_D(b) & 2 D+1 & \text { quadratic } & 1 \\[.05in]
D_D(b) & 2 D & \text { quadratic } & 0 \\[.025in]
& & \text { (Witt index } D) & \\[.05in]
{ }^2 D_{D+1}(b) & 2 D+2 & \begin{array}{l}
\text { quadratic }
\end{array} & 2 \\[.025in]
& & \text { (Witt index } D) & \\[.05in]
{ }^2 A_{2 D}(q) & 2 D+1 & \text { Hermitean }\left(b=q^2\right) & \frac{3}{2} \\[.05in]
{ }^2 A_{2 D-1}(q) & 2 D & \text { Hermitean }\left(b=q^2\right) & \frac{1}{2}
\end{array}
$$
For the rest of this section let $\Delta=(Y,\mathcal{F})$ denote a dual polar graph associated with one of the above nondegenerate forms, see \cite[Section 9.4]{BCN} for a detailed definition. Note that $\Delta$ is distance-regular with diameter $D$ and with intersection numbers 
\begin{equation}
	\label{eq:intnum}
	b_i = \frac{b^{i+e}(b^{D-i}-1)}{b-1}, \qquad c_i = \frac{b^i-1}{b-1}, 
	\qquad a_i = \frac{(b^e-1)(b^i-1)}{b-1},
\end{equation}
see \cite[Theorem 9.4.3]{BCN}. Observe that $\Delta$ is bipartite if and only if $e=0$. For the rest of this section we assume $\Delta$ is non-bipartite, that is $e \ne 0$.

Fix $x\in Y$, and let $T=T(x)$ denote the corresponding Terwilliger algebra of $\Delta$. Let $W$ denote an irreducible $T$-module with endpoint $r$ and diameter $d$. Then $W$ is thin (see \cite[Example 6.1]{Terpart3}). There exist a unique irreducible $T$-module with endpoint $0$, and this module has diameter $D$, see \cite[Lemma 3.6]{Terpart1}. We call this module the {\em trivial} $T$-module. The following result is a folklore, but for convenience of the reader we include a sketch of the proof.

\begin{proposition}
	\label{prop:long_module}
	 Let $\Delta=(Y,\mathcal{F})$ denote a non-bipartite dual polar graph. Fix $x\in Y$ and let $T=T(x)$ denote the corresponding Terwilliger algebra of $\Delta$. There exists an irreducible $T$-module with endpoint $1$ and diameter $D-1$.
\end{proposition}
\begin{proof}
	Let $W$ denote an irreducible $T$-module with endpoint $1$. Then the diameter of $W$ is either $D-2$ or $D-1$, see \cite[Theorem 10.7]{goter} and \cite[Theorem 11.4]{Ter}. 
	Pick $y \in Y$ such that $\partial(x,y)=1$. For $0 \le i,j \le D$ we define $D_j^i=D_j^i(x,y)$ by
$$
D_j^i = \{z \in Y \mid \partial(x,z)=i, \partial(y,z)=j\}.
$$ 
Note $D^i_j \ne \emptyset$ if and only if $|i-j| \le 1$ and $(i,j) \ne (0,0)$. It turns out that $\Delta$ is 1-homogeneous, see \cite[Theorem 1]{CN}. Consequently the characteristic vectors of the nonempty sets $D^i_j$ form a basis for a $T$-module $U$. Note that we have $\dim(\Es_0 U)=1$, $\dim(\Es_i U)=3$ for $1 \le i \le D-1$ and   $\dim(\Es_D U)=2$. Since every irreducible $T$-module is thin, it follows that $U$ is a direct sum of  the trivial $T$-module and two irreducible $T$-modules with endpoint 1, one of diameter $D-2$ and the other one of diameter $D-1$. This concludes the proof.
\end{proof}

Let  $\Delta_f$ denote the full bipartite graph of $\Delta$ with respect to $x$ and let $T_f=T_f(x)$ denote the corresponding Terwilliger algebra. By the above comments and Lemma \ref{lem:modules2} every irreducible $T$-module is also irreducible $T_f$-module. In particular, $T_f$ admits a thin irreducible $T_f$-module $W_0$ with endpoint $0$ and diameter $D$, and a thin irreducible $T_f$-module $W_1$ with endpoint $1$ and diameter $D-1$. Furthermore, it follows from \cite[Proposition 26.4]{w:dual} that $\Delta_f$ admits a uniform structure with respect to $x$ with corresponding parameters
\begin{equation}\label{eq:dp1}
 e_i^-=-\frac{b^2}{b+1} \ (2\le i\le \varepsilon), \quad  e_i^+=-\frac{b^{-1}}{b+1} \ (1\le i\le \varepsilon-1), \quad f_i=b^{e+D-1}  \ (1\le i\le \varepsilon).
\end{equation} 

\begin{theorem}\label{betrh}
With reference to Notation \ref{not}, let $\Ga$ denote the full bipartite graph of a non-bipartite dual polar graph $\Delta$ with respect to the vertex $x$. Pick scalars $\ths_0, \ths_1\in \RR$ such that $\ths_0\neq\ths_1$, and define 
\begin{equation}\label{rec}
\ths_{i+1}=\ths_1+(\ths_1-\ths_0)\frac{b^i-1}{b^i(b-1)} \ (1\le i\le \varepsilon-1).
\end{equation}
Then, $A^*=A^*(\ths_0,\ths_1, \ldots, \ths_{\varepsilon})$ is a dual adjacency matrix candidate of $\Ga$ with respect to $x$ with corresponding parameters $\beta=b+b^{-1}$ and $\rho=b^{e+D-1}(b+b^{-1}+2)$.
\end{theorem}
\begin{proof}
This easily follows from Remark \ref{check} using \eqref{eq:dp1}.
\end{proof}

\subsection{The eigenvalues}

With reference to Notation \ref{not}, let $\Ga$ denote the full bipartite graph of a non-bipartite dual polar graph $\Delta$ with respect to the vertex $x$. We will now study the adjacency matrix $A$ of $\Ga$ and display its eigenvalues. In order to do that, we first look at the action of the adjacency matrix $A$ on the irreducible $T$-modules.

Let $W$ denote an irreducible $T$-module with endpoint $r$ and diameter $d$. Recall the basis $\{w_i\}_{i=r}^{r+d}$ of $W$ from Theorem \ref{oldpaper}(iii). Let $A_{r,d}$ denote the matrix representing the action of $A$ on $W$ with respect to this basis. Then, by Theorem \ref{oldpaper}(iii) and since $A=R+L$,  we have 

\begin{equation}\label{rep}
A_{r,d}=\begin{pmatrix}
	0 &1                             \\
	x_{r+1} &0 &1 &  & \text{\huge0}\\
&\ddots  & \ddots & \ddots            \\
	 \text{\huge0} &  &x_{r+d-1}  & 0  &1    \\
	&  &   &x_{r+d}   & 0
\end{pmatrix}.
\end{equation}
In our next result we display the scalars $x_{r+i} \; (1\le i \le d)$.

\begin{theorem}{}
	\label{thm:scalars}
	With reference to Notation \ref{not}, let $\Ga$ denote the full bipartite graph of a non-bipartite dual polar graph $\Delta$ with respect to the vertex $x$. Let $W$ denote an irreducible $T$-module with endpoint $r$ and diameter $d$. Recall the basis $\{w_i\}_{i=r}^{r+d}$ of $W$ from Theorem \ref{oldpaper}(iii). Then,
	\begin{eqnarray}
		\label{eq:scalars}
		x_{r+i} =-b^{D+e} \frac{(b^{i-d-1}-1)(b^i-1)}{(b-1)^2} \quad (1 \le i \le d).
	\end{eqnarray}
\end{theorem}
\begin{proof}
	Using \eqref{eq:dp1}, it is straightforward to check that $\{x_{r+i}\}_{i=1}^d$ is a solution of the linear system \eqref{eq:linsys}.
\end{proof}

{Let $h_i$ denote the characteristic polynomial of the submatrix of $A_{r,d}$ obtained by removing the last $d+1-i$ rows and columns, for $1\leq i \leq d+1$.} Therefore, $h_{d+1}$ is the characteristic polynomial of $A_{r,d}$. It is straightforward  to check that the polynomials $h_i$ satisfy the recurrence formula
\begin{align*}
	h_{i+1}(t)&=t h_i(t)-x_{r+i}h_{i-1}(t),
\end{align*}
with $h_0(t)=1$, $h_1(t)=t$. The following theorem provides an explicit expression for the polynomial $h_{d+1}$.

\begin{theorem}
	\label{Krat}
	With reference to Notation \ref{not}, let $\Ga$ denote the full bipartite graph of a non-bipartite dual polar graph $\Delta$ with respect to the vertex $x$. Let $W$ denote an irreducible $T$-module with endpoint $r$ and diameter $d$ and consider the matrix $A_{r,d}$ as defined in \eqref{rep}. Then, the polynomials $h_i \; (0 \le i \le d+1)$, defined above, satisfy
	$$
h_{i}(t)=\left(\frac{b^{\frac{D+e+d}{2}}}{b-1}\right)^{i}H_{i}\left(b^{-\frac{D+e+d}{2}}(b-1)t\right),
	$$
	where $H_{i}(t)$ is the normalized dual $b$-Krawtchouk polynomial of degree $i$ with $c=-1$ and $N=d$ (see \cite[Equation~(9.11.4)]{koekoekbook} for a definition of these polynomials).
	In particular, the characteristic polynomial of $A_{r,d}$ is equal to
	$$
h_{d+1}(t)=\left(\frac{b^{\frac{D+e+d}{2}}}{b-1}\right)^{d+1}H_{d+1}\left(b^{-\frac{D+e+d}{2}}(b-1)t\right).
	$$
	
\end{theorem}
\begin{proof}
By \cite[Equation~(14.17.4)]{koekoekbook},  the normalized dual $b$-Krawtchouk polynomials (for  $c=-1$ and $N=d$) satisfy the recurrence formula
	\begin{align}\label{eq2}
		tH_i(t)=H_{i+1}(t)-b^{-d}\left(1-b^i\right)\left(1-b^{i-d-1}\right) H_{i-1}(t),
	\end{align}
with $H_0(t)=1$ and $H_1(t)=t$. Define the polynomials $g_i(t)$ as follows
$$
  g_i(t)=\left(\frac{b^{\frac{D+e+d}{2}}}{b-1}\right)^iH_i\left(b^{-\frac{D+e+d}{2}}(b-1)t\right).
$$
We will now show that $h_i(t)$ = $g_i(t)$ for $0 \le i \le d+1$. First, replace $t$ by $b^{-\frac{D+e+d}{2}}(b-1)t$ in \eqref{eq2} and multiply the obtained equality by $\left(\frac{b^{\frac{D+e+d}{2}}}{b-1}\right)^{i+1}$ on both sides in order to get
{\small	\begin{align*}
		&\left(b^{-\frac{D+e+d}{2}}(b-1)t\right)H_i\left(b^{-\frac{D+e+d}{2}}(b-1)t\right)\left(\frac{b^{\frac{D+e+d}{2}}}{b-1}\right)^{i+1}=\\&H_{i+1}\left(b^{-\frac{D+e+d}{2}}(b-1)t\right)\left(\frac{b^{\frac{D+e+d}{2}}}{b-1}\right)^{i+1}-b^{-d}\left(1-b^i\right)\left(1-b^{i-d-1}\right)H_{i-1}\left(b^{-\frac{D+e+d}{2}}(b-1)t\right)\left(\frac{b^{\frac{D+e+d}{2}}}{b-1}\right)^{i+1}.
	\end{align*}}
In other words,
$$
  tg_i(t)=g_{i+1}(t)-\frac{b^{D+e}}{(b-1)^2}\left(1-b^i\right)\left(1-b^{i-d-1}\right)g_{i-1}(t),
$$
or equivalently,
$$
	  tg_i(t)=g_{i+1}(t)+x_{r+i}g_{i-1}(t).
$$
Therefore, the polynomials $h_i$ and $g_i$ satisfy the same recurrence formula. Furthermore, note that 
$$
	g_0(t)=1\qquad\mbox{and} \qquad g_1(t)=\frac{b^{\frac{D+e+d}{2}}}{b-1}H_1\left(b^{-\frac{D+e+d}{2}}(b-1)t\right)=\frac{b^{\frac{D+e+d}{2}}}{b-1}b^{-\frac{D+e+d}{2}}(b-1)t=t.
$$
This implies that $h_i(t)=g_i(t)$ for all $0\le i \le d+1$. In particular,
$$
	h_{d+1}(t)=g_{d+1}(t)=\left(\frac{b^{\frac{D+e+d}{2}}}{b-1}\right)^{d+1}H_{d+1}\left(b^{-\frac{D+e+d}{2}}(b-1)t\right).
$$
This completes the proof. 
\end{proof}

Consider a linear map on a finite-dimensional vector space. This map is called \emph{multiplicity-free} whenever the map is diagonalizable, and each eigenspace has dimension one.
\begin{corollary}
	\label{eigen}
	With reference to Notation \ref{not}, let $\Ga$ denote the full bipartite graph of a non-bipartite dual polar graph $\Delta$ with respect to the vertex $x$. Let $W$ denote an irreducible $T$-module with endpoint $r$ and diameter $d$ and consider the matrix $A_{r,d}$ defined in \eqref{rep}. Then, $A_{r,d}$ is multiplicity-free with eigenvalues 
\begin{equation*}
	  \left\{\frac{b^{\frac{D+e}{2}}}{b-1}(b^{\frac{d}{2}-j}-b^{j-\frac{d}{2}})\,\,;\,\,0\leq j\leq d\right\}.
\end{equation*}
\end{corollary}

\begin{proof}
	By Theorem \ref{Krat}, the characteristic polynomial of $A_{r,d}$ is equal to  
\begin{equation}\label{dbK}
	  h_{d+1}(t)=\left(\frac{b^{\frac{D+e+d}{2}}}{b-1}\right)^{d+1}H_{d+1}\left(b^{-\frac{D+e+d}{2}}(b-1)t\right),
\end{equation}
	where $H_{d+1}(t)$ is the normalized dual $b$-Krawtchouk polynomial of degree $d+1$ with $c=-1$ and $N=d$  (cf. \cite[Equations~(14.17.1) and (14.17.4)]{koekoekbook}). It is known that this polynomial has zeros $b^{-j}-b^{-d+j}$, $0\leq j \leq d$; see for example \cite[Example 5.9]{Ter05} and \cite[Example 20.7]{Ter21}. Then it follows from \eqref{dbK} that
$$
h_{d+1}\left(\frac{b^{\frac{D+e+d}{2}}}{b-1}(b^{-j}-b^{-d+j})\right)=\left(\frac{b^{\frac{D+e+d}{2}}}{b-1}\right)^{d+1}H_{d+1}(b^{-j}-b^{-d+j})=0
$$
for $0\leq j\leq d$. Therefore, the roots of $h_{d+1}$ are $\left\{\frac{b^{\frac{D+e}{2}}}{b-1}(b^{\frac{d}{2}-j}-b^{j-\frac{d}{2}})\,\,;\,\,0\leq j\leq d\right\}$. As these roots are pairwise distinct and $A_{r,d}$ has dimension $(d+1) \times (d+1)$, this concludes the proof.
\end{proof}

\begin{theorem}
  \label{thm:eig}
  With reference to Notation \ref{not}, let $\Ga$ denote the full bipartite graph of a non-bipartite dual polar graph $\Delta$ with respect to the vertex $x$. Then the matrix $A$ is diagonalizable with eigenvalues $\theta_{i} \; (0 \le i \le 2D)$ 
  where 
  $$
    \theta_{i}=\frac{b^{\frac{D+e}{2}}}{b-1}\left(b^{\frac{D-i}{2}}-b^{\frac{i-D}{2}}\right)
  $$
\end{theorem}

\begin{proof}
	Observe that the matrix $A$ is diagonalizable because the standard module $V$ is a direct sum of irreducible $T$-modules and $A$ is diagonalizable on each of them. Let $\theta$ denote an eigenvalue of $A$. We first claim that $\theta=\theta_i$ for some $0 \le i \le 2D$. We mentioned that $V$ is a direct sum of irreducible $T$-modules. So, $\theta$ is an eigenvalue for the action of $A$ on some irreducible $T$-module $W$. Let $d$ denote the diameter of $W$. By Corollary \ref{eigen}, we have that $\theta=\frac{b^{\frac{D+e}{2}}}{b-1}(b^{\frac{d}{2}-j}-b^{j-\frac{d}{2}})$ for some $0 \le j \le d$. 
Set $i=D-d+2j$, and observe that $0\le D-d \le i \le D+d\le 2D$ since $0\le j \le d$ and $d \le D$. This shows that $\theta=\theta_{i}$.

Conversely, we now show that $\theta_i$ is an eigenvalue of $A$ for $0 \le i \le D$. Let $W_0$ denote the irreducible $T$-module with endpoint $0$ and diameter $D$. By Corollary \ref{eigen} and since the standard module is a direct sum of irreducible $T$-modules, we have that $\theta_i$ with $0 \le i \le 2D$, $i$ even, are eigenvalues of $A$. 
Let $W_1$ denote an irreducible $T$-module with endpoint $1$ and diameter $D-1$. By Corollary \ref{eigen},  the eigenvalues of the matrix $A_{1,D-1}$ are 
$$\frac{b^{\frac{D+e}{2}}}{b-1}\left(b^{\frac{D-1}{2}-j}-b^{j-\frac{D-1}{2}}\right)=\frac{b^{\frac{D+e}{2}}}{b-1}\left(b^{\frac{D-(2j+1)}{2}}-b^{\frac{(2j+1)-D}{2}}\right) \; (0 \le j \le D-1),$$ which are exactly the numbers $\theta_i$ with $0 \le i \le 2D$, $i$ odd. Since the standard module is a direct sum of irreducible $T$-modules, we have that $\theta_i$ with $0 \le i \le D$, $i$ odd, are eigenvalues of $A$. This finishes the proof.

\end{proof}

\subsection{A $Q$-polynomial structure for $A$}

With reference to Notation \ref{not}, in this subsection, we prove that the matrix $A^*$ is actually a dual adjacency matrix of $A$ (with respect to certain orderings of the primitive idempotents of $A$). To do that, first recall the eigenvalues  
\begin{equation}\label{Eigen}
        \theta_{i}=\frac{b^{\frac{D+e}{2}}}{b-1}\left(b^{\frac{D-i}{2}}-b^{\frac{i-D}{2}}\right) \ \ (0\leq i\leq 2D)
\end{equation}
of the matrix $A$. Let $V_i$ denote the eigenspace of $\theta_i$ and let $E_i$ denote the corresponding primitive idempotent. The following (i)--(iv) are well known.
\begin{itemize}
	\item[(i)] $E_iE_j=\delta_{ij}E_i \qquad (0\leq i,j\leq 2D)$, 
	\item[(ii)] $\displaystyle\sum_{i=0}^{2D} E_i=I$,
	\item[(iii)] $E_iA=\theta_iE_i =A E_i\qquad (0\leq i\leq 2D)$,
	\item[(iv)] $A=\displaystyle\sum_{i=0}^{2D}\theta_iE_i$.
\end{itemize}
For notational convenience, we assume that $E_i=0$ (and therefore also $V_i=\emptyset$) for $i < 0$ and $i > 2D$.

\begin{proposition}\label{zeroindex}
	With reference to Notation \ref{not} and the notation above, we have that $E_{i}A^*E_j=0$ for $0\leq i,j\leq 2D$ and $| i-j | \notin\{0,2\}$. 
\end{proposition}

\begin{proof}
	Multiplying the equation \eqref{3diagg} on the left by $E_i$ and on the right by $E_j$ and using the above-mentioned properties of primitive idempotents, we obtain
	\begin{equation*}
		(\theta_i-\theta_j)  \left(\theta_i^2+\theta_j^2-\beta\theta_i\theta_j-\rho\right)E_iA^*E_j=0.
	\end{equation*}
Using \eqref{Eigen} and $\beta=b+b^{-1}$, $\rho=b^{e+D-1}(b+b^{-1}+2)$ from Theorem \ref{betrh}, we rewrite the above equation as 
	\begin{equation*}
 \frac{b^{ \frac{3 e}{2}}}{(b-1)^3} (b^ {\frac{i}{2}} - b^{\frac{j}{2}-1})(b^{\frac{i}{2}}-b^{\frac{j}{2}}) (b^ {\frac{i}{2}} - b^{\frac{j}{2}+1})(b^{D-\frac{i}{2}-\frac{j}{2}-1} + 1)(b^{D-\frac{i}{2}-\frac{j}{2}} +1) (b^{D-\frac{i}{2}-\frac{j}{2}+1} +1) E_iA^*E_j=0.
	\end{equation*}
 Assume now that $| i-j | \notin\{0,2\}$. Since $b$ is neither zero nor a root of unity, the result follows.
\end{proof}

\begin{lemma}\label{tridact}
With reference to Notation \ref{not} and the notation above, the following holds:
	$$
	  A^* V_i\subseteq V_{i-2}+V_i+V_{i+2}\qquad (0\leq i\leq 2D).
	$$
\end{lemma}
\begin{proof}
	Using the above properties of the primitive idempotents $E_i$ together with the result from Proposition~\ref{zeroindex}, we get
	\begin{align*}
		A^*V_i = A^* E_i V=\sum_{j=0}^{2D}E_j A^* E_i V=& \ E_{i-2}A^*E_iV+E_iA^*E_iV_i+E_{i+2}A^*E_iV_i\\ \subseteq& \ V_{i-2}+V_i+V_{i+2}.
	\end{align*}
\end{proof}
We are now ready to prove our main result.
\begin{theorem}
	\label{thm:main}
	With reference to Notation \ref{not} and the notation above, the matrix $A^*$ is a dual adjacency matrix of $\Ga$ with respect to $x$ and with respect to the following orderings of primitive idempotents:
	\begin{itemize}
		\item[(i)] $E_0 < E_2 < \cdots < E_{2D} <E_1 <E_3 <\cdots <E_{2D-1}$;
		\item[(ii)] $E_1 < E_3 < \cdots < E_{2D-1} <E_0 <E_2 <\cdots <E_{2D}$.
	\end{itemize}
\end{theorem}
\begin{proof}
The result follows immediately from  Theorem \ref{betrh} and Lemma \ref{tridact}.
\end{proof}

\begin{corollary}
With reference to Notation \ref{not} and the notation above, the adjacency matrix $A$ of  $\Ga$ is $Q$-polynomial.
\end{corollary}

\begin{proof}
	Immediately from Theorem \ref{thm:main}.
\end{proof}

In \cite{FMMM-H} we demonstrated that the full bipartite
graph of a Hamming graph is $Q$-polynomial, with the corresponding parameters and a tridiagonal relation satisfying those described in Theorems~\ref{thm5} and \ref{thm6}. Consequently, the results in \cite{FMMM-H} serve as another application of the results presented in this paper.\\

We conclude this paper by posing an open problem. 
\begin{problem}
 With reference to Notation~\ref{not}, let $\Gamma$ denote the full bipartite graph of a non-bipartite dual polar graph $\Delta$ with respect to the vertex $x$. In Theorem~\ref{thm:eig} we displayed the eigenvalues of the adjacency matrix $A$ of the graph $\Gamma$. Find the multiplicities of these eigenvalues.   
\end{problem}
We would like to mention that in view of Corollary~\ref{eigen}, to find the multiplicities of the eigenvalues of the matrix $A$ it suffices to compute the multiplicity of an irreducible $T$-module $W$ with endpoint $r$ and diameter $d$ for every $0 \le r \le D$ and $0 \le d \le D$. Theorem~2.5(3) in \cite{OldTer} should be of great help.

\section{Acknowledgments}
Blas Fernández's research is partially funded by the Slovenian Research Agency, including research program P1-0285 and projects J1-2451, J1-3001, J1-4008, and J1-50000. Štefko Miklavič's research is partially supported by the Slovenian Research Agency through research program P1-0285 and projects N1-0140, N1-0159, J1-2451, N1-0208, J1-3001, J1-3003, J1-4008, J1-4084, and J1-50000. Giusy Monzillo's research is partially supported by the Ministry of Education, Science and Sport of the Republic of Slovenia through the University of Primorska Developmental Funding Pillar.


 \end{document}